\documentclass{jams-edited}

\usepackage{amssymb}

\usepackage{graphicx}

\usepackage[cmtip,all]{xy}

\usepackage{cite}
\usepackage{hyperref} 
\usepackage{color,soul} 
\usepackage{verbatim} 
 \usepackage{epsfig,epstopdf,verbatim,subfig,caption}
\usepackage[normalem]{ulem}  

\newtheorem{theorem}{Theorem}[section]

\theoremstyle{definition}
\newtheorem{definition}[theorem]{Definition}
\newtheorem{proposition}[theorem]{Proposition}

\theoremstyle{remark}
\newtheorem{remark}[theorem]{Remark}

\numberwithin{equation}{section}

\begin{document}

\title{A New Fractional Derivative with Classical Properties}

\dedicatory{The paper is dedicated to my school teacher \\ T.H. Gunaratna Silva} 

\author[U.N. Katugampola]{Udita N. Katugampola}
\address{Department of Mathematics, University of Delaware, Newark DE 19716, USA.}
\thanks{Department of Mathematical Sciences, University of Delaware, Newark DE 19716, USA}
\thanks{Email: uditanalin@yahoo.com \hfill Preprint submitted to J. Amer. Math. Soc.}
\email{uditanalin@yahoo.com}



\begin{abstract}
We introduce a new fractional derivative which obeys classical properties including: linearity, product rule, quotient rule, power rule, chain rule, vanishing derivatives for constant functions, the Rolle's Theorem and the Mean Value Theorem. The definition, 
\vspace{-.13cm}
\[
D^\alpha (f)(t) = \lim_{\epsilon \rightarrow 0} \frac{f(te^{\epsilon t^{-\alpha}}) - f(t)}{\epsilon},
\]
is the most natural generalization that uses the limit approach. For $0\leq \alpha < 1$, it generalizes the classical calculus properties of polynomials. Furthermore, if $\alpha = 1$, the definition is equivalent to the classical definition of the first order derivative of the function $f$. Furthermore, it is noted that there are $\alpha-$differentiable functions which are not differentiable. 
\end{abstract}

\maketitle




\section{Introduction}
The derivative of non-integer order has been an interesting research topic for several centuries. The idea was motivated by the question, ``What does it mean by $\frac{d^n f}{dx^n}$, if $n=\frac{1}{2}$ ?'', asked by L'Hospital in 1695 in his letters to Leibniz \cite{letter1,letter2, letter3}. Since then, the mathematicians tried to answer this question for centuries in several points of view. The outcomes are many folds. Various types of fractional derivatives were introduced: Riemann-Liouville, Caputo, Hadamard, Erd\'{e}lyi-Kober, Gr\"{u}nwald-Letnikov, Marchaud and Riesz are just a few to name \cite{udita1,udita2,key-9,key-8,key-2,key-12}.  Most of the fractional derivatives are defined via fractional integrals \cite{key-9}. Due to the same reason, those fractional derivatives inherit some non-local behaviors, which lead them to many interesting applications including memory effects and future dependence \cite{what}.

Among the inconsistencies of the existing fractional derivatives are:

\begin{enumerate}

\item[(1)] Most of the fractional derivatives except Caputo-type derivatives, do not satisfy $D_a^\alpha(1) = 0$ , if $\alpha$ is not a natural number. 

\item[(2)] All fractional derivatives do not obey the familiar Product Rule for two functions:
             \[
						    D_a^\alpha (fg) = fD_a^\alpha (g) + gD_a^\alpha(f).
						 \]
\item[(3)] All fractional derivatives do not obey the familiar Quotient Rule for two functions:
             \[
						    D_a^\alpha \Big(\frac{f}{g}\Big) = \frac{gD_a^\alpha (f) - fD_a^\alpha(g)}{g^2}.
						 \]
\item[(4)] All fractional derivatives do not obey the Chain Rule:
             \[
						    D_a^\alpha (f \circ g)(t) = f^{(\alpha)}\big(g(t)\big) \, g^{(\alpha)}(t). 
						 \]	
\item[(5)] Fractional derivatives do not have a corresponding Rolle's Theorem.

\item[(6)] Fractional derivatives do not have a corresponding Mean Value Theorem.	
					
\item[(7)] All fractional derivatives do not obey: $D_a^\alpha D_a^\beta f = D_a^{\alpha + \beta} f$, in general.

\item[(8)] The Caputo definition assumes that the function $f$ is differentiable. 					

\end{enumerate}

To overcome some of these and other difficulties, Khalil et al. \cite{Khalil}, came up with an interesting idea that extends the familiar limit definition of the derivative of a function given by the following. 
\begin{definition}\cite{Khalil}
Let $f:[0, \infty) \rightarrow \mathbb{R}$ and $t > 0$. Then the ``conformable fractional derivative'' of $f$ of order $\alpha$ is defined by,
\begin{equation}
    f^{(\alpha)}(t) = \lim_{\epsilon \rightarrow 0} \frac{f(t + \epsilon t^{1-\alpha}) - f(t)}{\epsilon},
		\label{def1}
\end{equation}
for $t >0, \; \alpha \in (0, 1)$. If $f$ is $\alpha-$differentiable in some $(0, a), \; a >0$, and $\lim_{t \rightarrow 0^+} f^{(\alpha)}(t)$ exists, then define
\[
f^{(\alpha)}(0) = \lim_{t \rightarrow 0^+} f^{(\alpha)}(t). 
\]
\end{definition}

As a consequence of the above definition, the authors in \cite{Khalil}, showed that the $\alpha-$derivative in (\ref{def1}), obeys the Product rule, Quotient rule and has results similar to the Rolle's Theorem and the Mean Value Theorem in classical calculus. The purpose of this work is to further generalize the results obtained in \cite{Khalil} and introduce a new fractional derivative as the most natural extension of the familiar limit definition of the derivative of a function $f$ at a point. 

\section{New Fractional Derivative}
In this section, we give the main definition of the paper and obtain several results that are close resemblance of the results found in classical calculus. We prove that the new fractional derivative satisfies the Product rule, Quotient rule, Chain rule and have results which are natural extensions of the Rolle's theorem and the Mean Value Theorem. To this end, we start with the following definition, which is a generalization of the limit definition of the derivative. 
\begin{definition} 
Let $f:[0, \infty) \rightarrow \mathbb{R}$ and $t > 0$. Then the \emph{fractional derivative} of $f$ of order $\alpha$ is defined by,
\begin{equation}
    \mathcal{D}^\alpha (f)(t) = \lim_{\epsilon \rightarrow 0} \frac{f(te^{\epsilon t^{-\alpha}}) - f(t)}{\epsilon},
		\label{df}
\end{equation}
for $t >0, \; \alpha \in (0, 1)$. If $f$ is $\alpha-$differentiable in some $(0, a), \; a >0$, and $\lim_{t \rightarrow 0^+} \mathcal{D}^\alpha (f)(t)$ exists, then define
\[
\mathcal{D}^\alpha (f)(0) = \lim_{t \rightarrow 0^+} \mathcal{D}^\alpha (f)(t). 
\]
\end{definition}

The first result is the generalization of Theorem 2.1 of \cite{Khalil}. 
\begin{theorem}
If a function $f:[0, \infty) \rightarrow \mathbb{R}$ is $\alpha-$differentiable at $a >0, \; \alpha \in (0, 1]$, then $f$ is continuous at $a$. 
\end{theorem}
\begin{proof}
Since $f(ae^{\epsilon a^{-\alpha}})-f(a) = \frac{f(ae^{\epsilon a^{-\alpha}})-f(a)}{\epsilon} \epsilon$, we have
\[
  \lim_{\epsilon \rightarrow 0} \big[f(ae^{\epsilon a^{-\alpha}})-f(a)\big] = \lim_{\epsilon \rightarrow 0} \frac{f(ae^{\epsilon a^{-\alpha}})-f(a)}{\epsilon} .\lim_{\epsilon \rightarrow 0}\epsilon. 
\]
Let $h= \epsilon a^{1-\alpha} + O(\epsilon^2)$. Then, $\lim_{h \rightarrow 0} \big[f(a+h)-f(a)\big] = \mathcal{D}^\alpha (f)(a) \cdot 0 = 0$, which, in turn, implies that $\lim_{h \rightarrow 0} f(a+h) = f(a)$. This completes the proof. 
\end{proof}
The following is the main result of this paper. 

\vspace{.2cm} 

\begin{theorem} \label{thm3}
Let $\alpha \in (0, 1]$ and $f, g$ be $\alpha-$differentiable at a point $t>0$. Then,
\begin{enumerate}
	\item[(1)] $\mathcal{D}^\alpha \big[af +bg\big] = a\mathcal{D}^\alpha (f) + b\mathcal{D}^\alpha (g)$, for all $a, b \in \mathbb{R}$.
	
	\item[(2)] $\mathcal{D}^\alpha \big(t^n \big) = nt^{n-\alpha}$ for all $n \in \mathbb{R}$. 
	
	\item[(3)] $\mathcal{D}^\alpha \big(C \big) = 0,$ for all constant functions, $f(t) = C.$
	
	\item[(4)] $\mathcal{D}^\alpha \big(fg) = f\mathcal{D}^\alpha (g) + g\mathcal{D}^\alpha (f)$.
	
	\item[(5)] $\mathcal{D}^\alpha \big(\frac{f}{g}) = \frac{g\mathcal{D}^\alpha (f) - f\mathcal{D}^\alpha (g)}{g^2}$.
	
	\item[(6)] $\mathcal{D}^\alpha (f \circ g)(t) = f^\prime\big(g(t)\big) \, \mathcal{D}^\alpha g(t)$, \;\mbox{for} $f$ \mbox{differentiable at }$g(t)$. 
	
	\item[(7)] If, in addition, $f$ is differentiable, then $\mathcal{D}^\alpha \big(f)(t) = t^{1-\alpha}\frac{df}{dt}(t).$
\end{enumerate}
\end{theorem}

\begin{proof}
 Part (1) and (3) follow directly from the definition. We shall prove (2), (4), (6) and (7), since the proofs are different from what appear in \cite{Khalil}. 
Now, for fixed $\alpha \in (0, 1]$, $n \in \mathbb{R}$ and $t>0$, we have
\begin{align}
  \mathcal{D}^\alpha \big(t^n\big) &= \lim_{\epsilon \rightarrow 0} \frac{(te^{\epsilon t^{-\alpha}})^n- t^n}{\epsilon} \nonumber  \\
	                                &= t^n\lim_{\epsilon \rightarrow 0} \frac{e^{n\epsilon t^{-\alpha}}-1}{\epsilon} \nonumber \\ 
																	&= t^n\lim_{\epsilon \rightarrow 0} \frac{\epsilon n t^{-\alpha}+\frac{\epsilon^2 n^2 t^{-2\alpha}}{2!} +\cdots}{\epsilon} \nonumber \\ 
																	&= t^n\lim_{\epsilon \rightarrow 0} \frac{\epsilon n t^{-\alpha}+ O(\epsilon^2)}{\epsilon} \label{eq1} \\
																	&=nt^{n-\alpha}. \nonumber
\end{align}
This completes the proof of (2). Then, we shall prove (4). To this end, since $f, g$ are $\alpha-$differentiable at $t >0,$ note that,
\begin{align}
  \mathcal{D}^\alpha \big(fg\big)(t) &= \lim_{\epsilon \rightarrow 0} \frac{f(te^{\epsilon t^{-\alpha}})g(te^{\epsilon t^{-\alpha}})-f(t)g(t)}{\epsilon} \nonumber  \\
	                                &= \lim_{\epsilon \rightarrow 0} \frac{f(te^{\epsilon t^{-\alpha}})g(te^{\epsilon t^{-\alpha}})-f(t)g(te^{\epsilon t^{-\alpha}})+f(t)g(te^{\epsilon t^{-\alpha}})-f(t)g(t)}{\epsilon} \nonumber  \\
																	&= \lim_{\epsilon \rightarrow 0} \Big[\frac{f(te^{\epsilon t^{-\alpha}})-f(t)}{\epsilon}\cdot g(te^{\epsilon t^{-\alpha}})\Big] + f(t)\lim_{\epsilon \rightarrow 0} \frac{g(te^{\epsilon t^{-\alpha}})-g(t)}{\epsilon} \nonumber \\
																	&= \mathcal{D}^\alpha \big(f\big)(t)\lim_{\epsilon \rightarrow 0} g(te^{\epsilon t^{-\alpha}}) +f(t)\mathcal{D}^\alpha \big(g\big)(t). \nonumber
\end{align}															
Since $g$ is continuous at $t$, $\lim_{\epsilon \rightarrow 0} g(te^{\epsilon t^{-\alpha}}) = g(t).$ This completes the proof of (4). The proof of (5) is similar and is left for the reader to verify. Next, we prove (6) in two different approaches. First, suppose $u=g(t)$ is $\alpha-$differentiable at a point $a>0$ and $y=f(u)$ is \textsl{differentiable} at a point $b=g(a)>0$. Let $\epsilon>0$, and $\Delta y = f(be^{\epsilon b^{-\alpha}})-f(b)$. Since $be^{\epsilon b^{-\alpha}} = b + \epsilon b^{1-\alpha}+O(\epsilon^2) = b+\Delta u$, where $\Delta u = \epsilon b^{1-\alpha}+O(\epsilon^2)$, we have
\begin{equation}
	\Delta y = D^1 f(b) \Delta u + \epsilon_1 \Delta u   \label{eq2}
\end{equation}	
where $\epsilon_1 \rightarrow 0$ as $\Delta u \rightarrow 0$. Thus, $\epsilon_1$ is a continuous function of $\Delta u$ if we define $\epsilon_1$ to be $0$ when $\Delta u =0$. Now, if $\Delta t$ is an increment in $t$ and $\Delta u $ and $\Delta y$ (with the possibility of both being equal to $0$) are corresponding increments in $u$ and $y$, respectively. Then we may write, using Equation (\ref{eq2}),
\begin{equation}
	\Delta u = \mathcal{D}^\alpha g(a) \Delta t + \epsilon_2 \Delta t   \label{eq3}
\end{equation}	 
where $\epsilon_2 \rightarrow 0$ as $\Delta t \rightarrow 0$, and
\begin{align}
	\Delta y &= \Big[D^1 f(b)+ \epsilon_1 \Big] \Delta u  \nonumber \\
	         &= \Big[D^1 f(b) + \epsilon_1 \Big] \cdot \Big[\mathcal{D}^\alpha g(a) + \epsilon_2 \Big] \Delta t   \nonumber 
\end{align}
where both $\epsilon_1 \rightarrow 0$ and $\epsilon_2 \rightarrow 0$ as $\Delta t \rightarrow 0$. Taking $\Delta t = \epsilon$, we now have,
\begin{align*}
        \mathcal{D}^\alpha (f \circ g)(t) &= \lim_{\epsilon \rightarrow 0} \frac{\Delta y}{\epsilon} \\
				                        &= \lim_{\Delta t \rightarrow 0} \Big[D^1 f(b) + \epsilon_1 \Big] \cdot \Big[\mathcal{D}^\alpha g(a) + \epsilon_2 \Big] \\
																&=D^1 f(b)\mathcal{D}^\alpha g(a) = f^\prime(g(a))\mathcal{D}^\alpha g(a).
\end{align*}
It can be seen that Equation~\ref{eq3}, can not be written in the form of Equation~\ref{eq2}. This makes it necessitates to use the $\alpha-$derivative in Equation~\ref{eq3}. 

Now, we prove the result following a standard limit-approach. To this end, in one hand, if the function $g$ is constant in a neighborhood containing $a$, then $\mathcal{D}^\alpha \big(f\circ g\big)(a) = 0$. On the other hand, assume that the function $g$ is non-constant in the neighborhood of $a$. In this case, we can find an $\epsilon_0 >0$ such that, $g(x_1) \neq g(x_2)$ for any $x_1, x_2 \in (a-\epsilon_0, a+\epsilon_0)$. Now, since $g$ is continuous at $a$, for $\epsilon$ sufficiently small, we have
\begin{align*}
   \mathcal{D}^\alpha \big(f\circ g\big)(a) &= \lim_{\epsilon \rightarrow 0} \frac{f\big(g(ae^{\epsilon a^{-\alpha}})\big)-f\big(g(t)\big)}{\epsilon} \\
	&=\lim_{\epsilon \rightarrow 0} \frac{f\big(g(ae^{\epsilon a^{-\alpha}})\big)-f\big(g(a)\big)}{g(ae^{\epsilon a^{-\alpha}})-g(a)}\cdot \frac{g(ae^{\epsilon a^{-\alpha}})-g(a)}{\epsilon} \\
	&=\lim_{\epsilon \rightarrow 0} \frac{f\big(g(a)+\epsilon_1\big)-f\big(g(a)\big)}{\epsilon_1}\cdot \frac{g(ae^{\epsilon a^{-\alpha}})-g(a)}{\epsilon}, \; \mbox{ where } \; \epsilon_1 \rightarrow 0 \mbox{ as } \epsilon \rightarrow 0, \\
	&=\lim_{\epsilon_1 \rightarrow 0} \frac{f\big(g(a)+\epsilon_1\big)-f\big(g(a)\big)}{\epsilon_1} \cdot \lim_{\epsilon \rightarrow 0} \frac{g(ae^{\epsilon a^{-\alpha}})-g(a)}{\epsilon} \\
	&=f^\prime\big(g(a)\big) \mathcal{D}^\alpha g(a), 
\end{align*}
for $a >0$. This establishes the Chain Rule. 

To prove part (7), we use a similar line of argument as in Equation (\ref{eq1}). Thus, taking $h=\epsilon t^{1-\alpha} \big(1+ O(\epsilon)\big)$, we have
\begin{align*}
    \mathcal{D}^\alpha (f)(t) &= \lim_{\epsilon \rightarrow 0} \frac{f(te^{\epsilon t^{-\alpha}})-f(t)}{\epsilon} \\
		                &= \lim_{\epsilon \rightarrow 0} \frac{f(t+\epsilon t^{1-\alpha}+O(\epsilon^2))-f(t)}{\epsilon} \\
										&= \lim_{\epsilon \rightarrow 0} \frac{f(t+h)-f(t)}{\frac{ht^{\alpha -1}}{1+O(\epsilon)}} \\
										&=t^{1-\alpha}\frac{df}{dt}(t),
\end{align*}
\noindent since, by the assumption, $f$ is differentiable at $t>0$. This completes the proof of the theorem. 
\end{proof}
\begin{remark} The result similar to (6) does not appear in \cite{Khalil}. The argument used in Equation (\ref{eq1}) was/will be utilized in several occasions in the paper. \end{remark}

The following theorem lists $\alpha-$fractional derivative of several familiar functions. The results are identical to that of conformable fractional derivative discussed in \cite{Khalil}, which will be shown to be a special case of $\alpha-$fractional derivative defined in (\ref{df}).  
\begin{theorem} \label{thm4} Let $a, n \in \mathbb{R}$ and $\alpha \in (0, 1]$. Then, we have the following results.
\begin{enumerate}
\item[(a)] $\mathcal{D}^\alpha (t^n) = n t^{n-\alpha}$.
\item[(b)] $\mathcal{D}^\alpha (1) = 0.$
\item[(c)] $\mathcal{D}^\alpha (e^{ax}) = ax^{1-\alpha}e^{ax}.$
\item[(d)] $\mathcal{D}^\alpha (\sin ax) = ax^{1-\alpha}\cos ax.$
\item[(e)] $\mathcal{D}^\alpha (\cos ax) = -ax^{1-\alpha}\sin ax.$
\item[(f)] $\mathcal{D}^\alpha (\frac{1}{\alpha}t^\alpha) = 1.$
\end{enumerate}
\end{theorem} 
It is easy to see from part (6) of Theorem~\ref{thm3} that we have rather unusual results given in the next theorem. 
\begin{theorem} \label{thm5} Let $\alpha \in (0, 1]$ and $t >0$. Then, 
\begin{enumerate}
\item[(i)] $\mathcal{D}^\alpha (\sin \frac{1}{\alpha}t^\alpha ) = \cos \frac{1}{\alpha}t^\alpha.$
\item[(ii)] $\mathcal{D}^\alpha (\cos \frac{1}{\alpha}t^\alpha) = -\sin \frac{1}{\alpha}t^\alpha.$
\item[(iii)] $\mathcal{D}^\alpha (e^{\frac{1}{\alpha}t^\alpha}) = e^{\frac{1}{\alpha}t^\alpha}.$
\end{enumerate}
\end{theorem} 

Theorem \ref{thm5} suggests that there is a pseudo-invariant space corresponding to the $\alpha-$fractional derivative on which the $sine$, $cosine$ and the exponential function, $e^x$ behave as they are having familiar classical derivatives. Due to the Fourier theory, this in turn, implies that any function which possesses a Fourier expansion behaves as having classical derivatives. This would be an interesting topic for further research. 

\begin{figure}[h]
 \centering
   \subfloat[$\nu$= 2.0]{\includegraphics[height = 6.0cm,width=6.3cm, trim = 0 20 0 0]{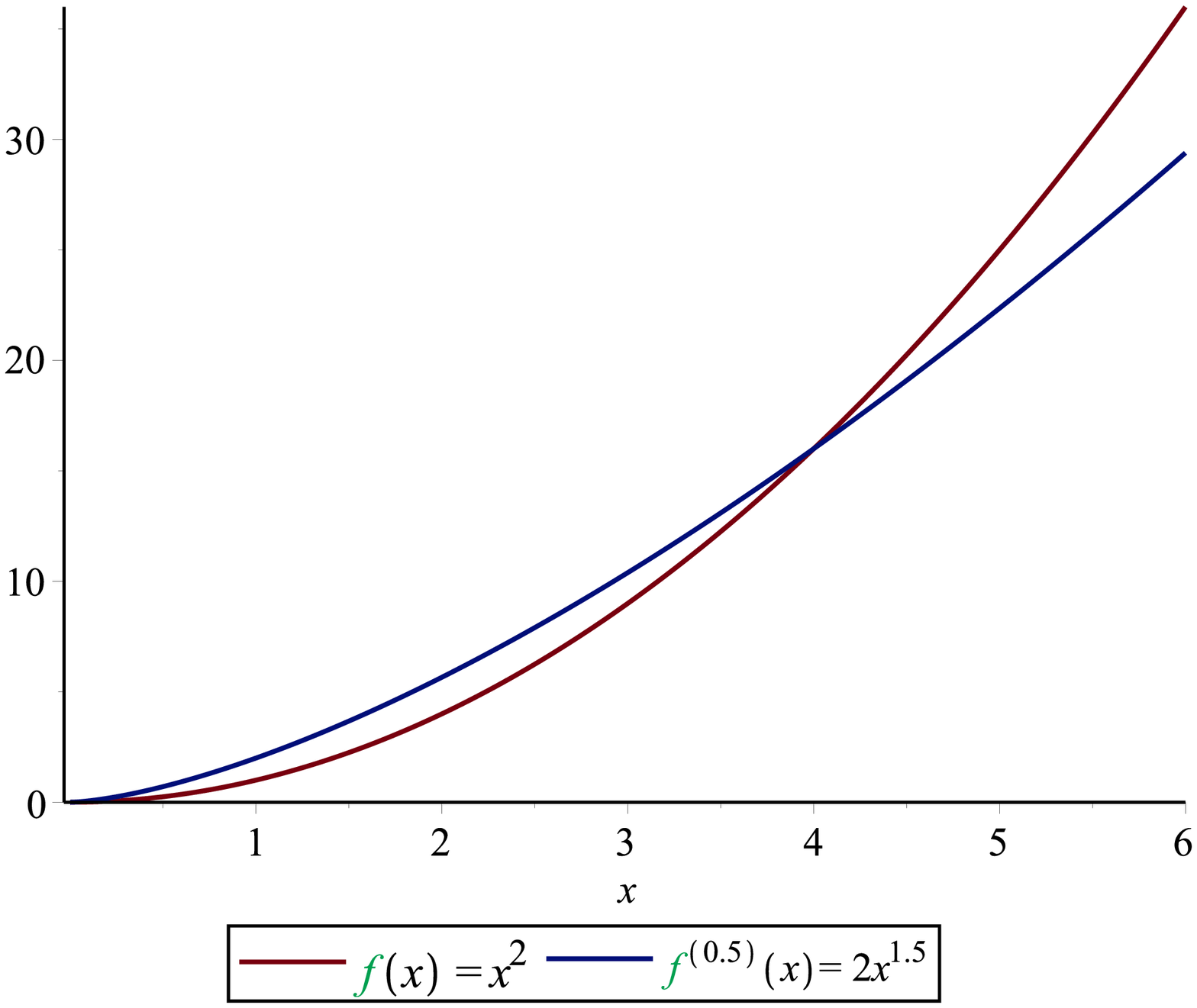}}
	\subfloat[$\nu$= 1.0]{\includegraphics[height=5.7cm,width=6cm, trim=0 0 0 25]{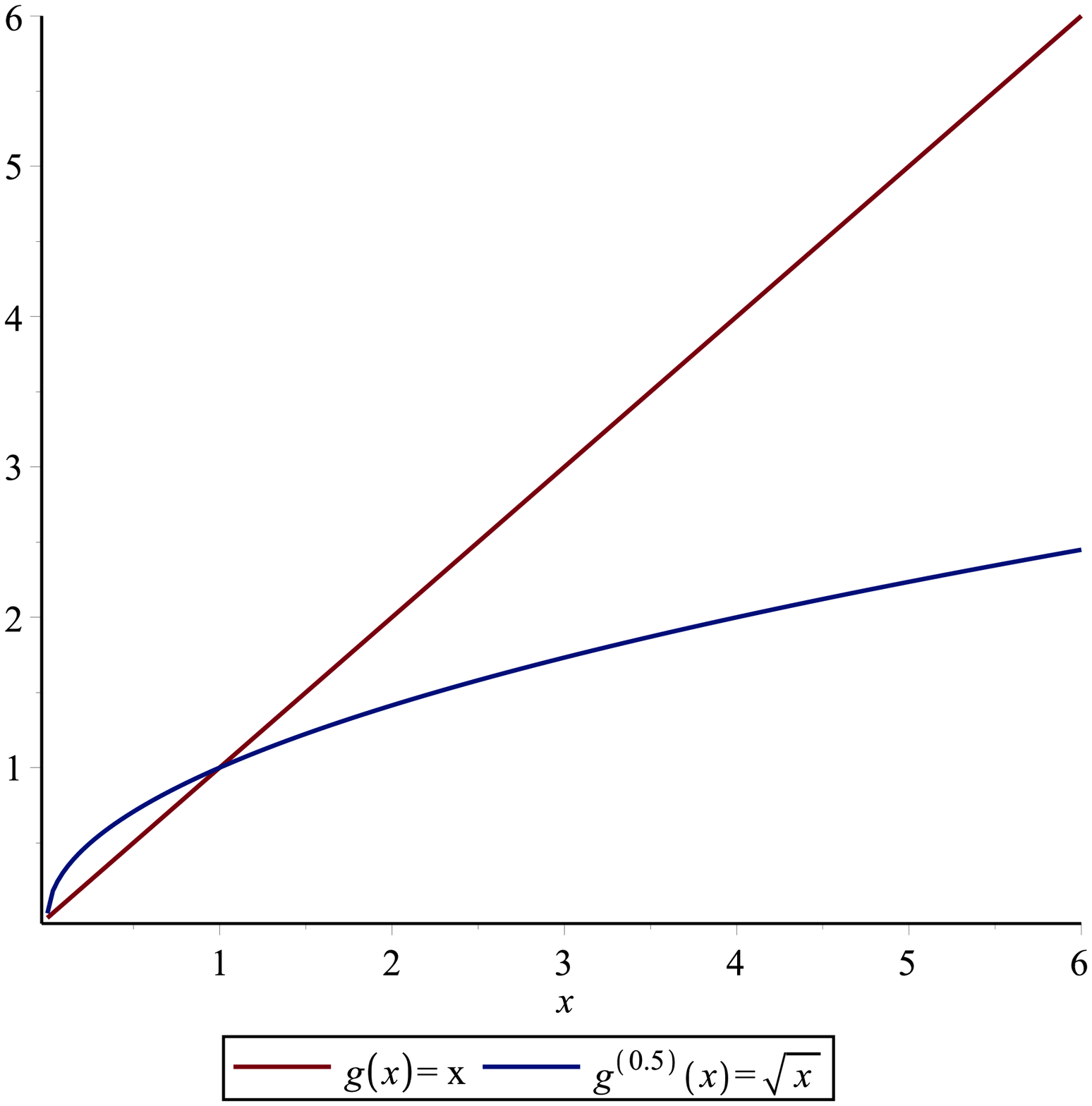}}
\caption{$\frac{1}{2}-$fractional derivatives of power function $f(x)=x^\nu$.}
\label{fig:deri}
\end{figure}


In \cite{udita1,udita2}, the author introduced fractional integrals and derivatives, which generalize the Riemann-Liouville and the Hadamard fractional integrals and derivatives to a single form. In the next subsection, we show that the results obtained by Khalil et at. \cite{Khalil} are special cases of the more general extension discussed in this paper. 

\subsection{Further Generalizations}
Let us first define a truncated exponential function given by,
\[
	e_k^x = \sum_{i=0}^k \frac{x^i}{i!} \label{e-k}
\]
With the help of this definition, we will define another fractional derivative given below.

\begin{definition}
Let $f:[0, \infty) \rightarrow \mathbb{R}$ and $t > 0$. Then the \emph{fractional derivative} of $f$ of order $\alpha$ is defined by,
\begin{equation}
    \mathcal{D}_k^\alpha (f)(t) = \lim_{\epsilon \rightarrow 0} \frac{f(te_k^{\epsilon t^{-\alpha}}) - f(t)}{\epsilon},
		\label{def2}
\end{equation}
for $t >0, \; \alpha \in (0, 1)$. If $f$ is $\alpha-$differentiable in some $(0, a), \; a >0$, and $\lim_{t \rightarrow 0^+} \mathcal{D}_k^\alpha (f)(t)$ exists, then define
\[
\mathcal{D}_k^\alpha (f)(0) = \lim_{\epsilon \rightarrow 0} \mathcal{D}_k^\alpha (f)(t). 
\]
\end{definition}
Now, It is easy to see that 
\begin{equation}
   \mathcal{D}_1^\alpha (f)(t) = \lim_{\epsilon \rightarrow 0} \frac{f(t + \epsilon t^{1-\alpha}) - f(t)}{\epsilon}, \label{ndef1}
\end{equation}
and
\begin{equation}
   \mathcal{D}_\infty^\alpha (f)(t) = \lim_{\epsilon \rightarrow 0} \frac{f(te^{\epsilon t^{-\alpha}}) - f(t)}{\epsilon}, \label{ndef2}
\end{equation}
where the expression on the right-hand-side of (\ref{ndef1}) is the conformable fractional derivative defined in (\ref{def1}) in \cite{Khalil}, while the one on (\ref{ndef2}) is the $\alpha-$fractional derivative defined in this paper. Furthermore, if $\alpha = 1$ the Equation (\ref{ndef1}) becomes the classical definition of the first derivative of a function $f$ at a point $t$. While, noting that $te^{\epsilon t^{-\alpha}} = t+ \epsilon +O(\epsilon^2)$, we may also show that the Equation (\ref{ndef2}) is equivalent to the classical definition of the first derivative of a function $f$. These observations further suggest that there are corresponding results similar to the Rolle's theorem and the Mean Value theorem. Next, we give those two results.
\begin{theorem}[Rolle's theorem for $\alpha-$Fractional Differentiable Functions]
Let $a>0$ and $f:[a, b] \rightarrow \mathbb{R}$ be a function with the properties that
\begin{enumerate}
\item[(1)] $f$ is continuous on $[a, b]$,
\item[(2)] $f$ is $\alpha-$differentiable on $(a, b)$ for some $\alpha \in (0, 1)$,
\item[(3)] $f(a)=f(b)$.
\end{enumerate}
Then, there exists $c\in(a,b)$, such that $\mathcal{D}^\alpha (f)(c) = 0.$
\end{theorem}
\begin{proof}
We prove this using contradiction. Since $f$ is continuous on $[a, b]$ and $f(a)=f(b)$, there is $c\in(a, b)$, at which the function has a local extrema. 
Then, 
\[
  \mathcal{D}^\alpha f(c)=\lim_{\epsilon \rightarrow 0^-} \frac{f(ce^{\epsilon c^{-\alpha}}) - f(c)}{\epsilon}=\lim_{\epsilon \rightarrow 0^+} \frac{f(ce^{\epsilon c^{-\alpha}}) - f(c)}{\epsilon}.
\]
But, the two limits have opposite signs. Hence, $\mathcal{D}^\alpha f(c) =0.$ 
\end{proof}

\begin{theorem}[Mean Value Theorem for $\alpha-$Fractional Differentiable Functions]
Let $a>0$ and $f:[a, b] \rightarrow \mathbb{R}$ be a function with the properties that
\begin{enumerate}
\item[(1)] $f$ is continuous on $[a, b]$,
\item[(2)] $f$ is $\alpha-$differentiable on $(a, b)$ for some $\alpha \in (0, 1)$,
\end{enumerate}
Then, there exists $c\in(a,b)$, such that $\mathcal{D}^\alpha (f)(c) = \frac{f(b)-f(a)}{\frac{1}{\alpha}b^\alpha-\frac{1}{\alpha}a^\alpha}.$
\end{theorem}
\begin{proof}
Consider the function,
\[
   g(x) = f(x) -f(a) - \frac{f(b)-f(a)}{\frac{1}{\alpha}b^\alpha-\frac{1}{\alpha}a^\alpha}\Big(\frac{1}{\alpha}x^\alpha-\frac{1}{\alpha}a^\alpha\Big).
\]
Then, the function $g$ satisfies the conditions of the fractional Rolle's theorem. Hence, there exists $c\in(a,b)$, such that $\mathcal{D}^\alpha (g)(c) =0$. Using the fact that $\mathcal{D}^\alpha (\frac{1}{\alpha}x^\alpha) =1$, the result follows. 
\end{proof}
As a result of the fractional Mean Value Theorem, we also have the following proposition. 

\begin{proposition}
Let $f:[a, b] \rightarrow \mathbb{R}$ be $\alpha-$differentiable for some $\alpha \in (0, 1)$. Suppose that $\mathcal{D}^\alpha f $ is bounded on $[a, b]$. Then, if 
\begin{enumerate}
	\item[(a)] $a>0$, or
	\item[(b)] $\mathcal{D}^\alpha f $ is continuous at either $a$ or $b$, 
\end{enumerate}
the function $f$ is uniformly continuous on $[a, b]$, and hence $f$ is bounded. 
\end{proposition}

Figure 1 depicts the graphs of the $\alpha-$fractional derivatives of the power function $f(x)=x^\nu$ for $\nu = 1$ and $2$. It is also interesting to note that the function $f(x) = \sqrt{x}$ has a constant derivative equal to $\frac{1}{2}$, and not equal to $\sqrt{\pi}/2$ as in the case of the Riemann-Liouville derivative. This also suggests that the new derivative should have a rather different interpretation in the geometrical sense, which has yet to be discovered. 

Further, it can be seen that a function could be $\alpha-$differential at a point but not differentiable in the classical sense. (cf. page 67 of \cite{Khalil}). For example, the function $f(t) = 3t^{\frac{1}{3}}$ is everywhere $\frac{1}{3}-$differentiable and has the derivative of $1$ while it is not differentiable at $t=0$. The similar results are there for the Riemann-Liouville and the Caputo derivatives. 

Next, we consider the possibility of $\alpha \in (n, n+1]$, for some $n\in \mathbb{N}$. We have the following definition.
\begin{definition} \label{def5}
Let $\alpha \in (n, n+1]$, for some $n\in \mathbb{N}$ and $f$ be an $n-$differentiable at $t >0$. Then the $\alpha-$fractional derivative of $f$ is defined by
\[
   \mathcal{D}^\alpha f (t) = \lim_{\epsilon \rightarrow 0} \frac{f^{(n)}(te^{\epsilon t^{n-\alpha}})-f^{(n)}(t)}{\epsilon},
\]
if the limit exists. 
\end{definition}
As a direct consequence of Definition \ref{def5} and part (7) of theorem \ref{thm3}, we can show that $\mathcal{D}^\alpha f (t) = t^{n+1-\alpha}f^{(n+1)}(t)$, where $\alpha \in (n, n+1]$ and $f$ is $(n+1)-$differentiable at $t>0$.

Now, the question whether the fractional derivative $\mathcal{D}^\alpha f$ has a corresponding $\alpha-$fractional integral will be answered next. For simplicity, we shall only consider the class of continuous functions here. The results can easily be extended to more general settings and will be discussed in forthcoming papers. 

\section{Fractional Integral}
As in the works of \cite{Khalil}, it is interesting to note that, in spite of the variation of the definitions of the fractional derivatives, we can still adopt the same definition of the fractional integral here due to the fact that we obtained similar results in Theorem \ref{thm4} as of the results (1)--(6) and (i)--(iii) in \cite{Khalil}. So, we have the following definition.
\begin{definition}[Fractional Integral] 
Let $a\geq 0$ and $t \geq a$. Also, let $f$ be a function defined on $(a, t]$ and $\alpha \in \mathbb{R}$. Then, the \emph{$\alpha-$fractional integral} of $f$ is defined by,
\begin{equation}
   \mathcal{I}^\alpha_a (f)(t) = \int_a^t \frac{f(x)}{x^{1-\alpha}} \, dx
	\label{int1}
\end{equation}
if the Riemann improper integral exists. 
\end{definition}

It is interesting to observe that the $\alpha-$fractional derivative and the $\alpha-$fractional integral are inverse of each other as given in the next result. 
\begin{theorem}[Inverse property]
Let $a \geq 0$,  and $\alpha \in (0, 1)$. Also, let $f$ be a continuous function such that $\mathcal{I}^\alpha_a f$ exists. Then
\[
   \mathcal{D}^\alpha\Big(\mathcal{I}^\alpha_a (f)\Big) (t) = f(t), \quad \quad \mbox{for} \;\; t \geq a. 
\]
\end{theorem}
\begin{proof} The proof is a direct consequence of the fundamental theorem of calculus. Since $f$ is continuous, $\mathcal{I}^\alpha_a f$ is clearly differentiable. Therefore, using part (7) of theorem \ref{thm3}, we have
\begin{align*}
   \mathcal{D}^\alpha\Big(\mathcal{I}^\alpha_a (f)\Big) (t) &= t^{\alpha -1}\frac{d}{dt} \mathcal{I}^\alpha_a (f) (t), \\
	                                                          &= t^{\alpha -1}\frac{d}{dt}\int_a^t \frac{f(x)}{x^{1-\alpha}} \, dx, \\
																														&= t^{\alpha -1}\frac{f(t)}{t^{1-\alpha}}, \\
  																													&= f(t).																							
\end{align*}
This completes the proof. 
\vspace{-.15cm}
\end{proof}		

The applications of this $\alpha-$integral and derivative will be discussed in a forthcoming paper. Several interesting applications of a similar fractional derivative defined in (\ref{def1}) and the corresponding fractional integral appear in \cite{Khalil}, the major reference of this paper. 

We conclude the paper with the following questions, which have yet to be answered.



\begin{enumerate}
  \item[(i)] What is the physical meaning or geometric interpretation of the $\alpha-$derivative?
	\item[(ii)]Is there a similarity between the $\alpha-$derivative and the \emph{G\^{a}teaux derivative}?
	\item[(iii)]One of the limitations of this version of the fractional derivative is that it assumes that the variable $t >0$. So the question is whether we can relax this condition on a special class of functions, if so, what is it?
\end{enumerate}











\bibliographystyle{abbrv}

\end{document}